\documentclass[reqno,A4paper]{amsart}
\usepackage{amsmath}
\usepackage{amssymb}
\usepackage{amsthm}
\usepackage{enumerate}
\usepackage{mathrsfs} 
\usepackage{eqlist}
\usepackage{array}
\usepackage[unicode]{hyperref}
\theoremstyle{plain}
\newtheorem{theorem}{Theorem}[section]

\newtheorem{lemma}{Lemma}[section]

\theoremstyle{definition}
\newtheorem{definition}{Definition}
\newtheorem{remark}{\textup{Remark}} 

\newtheorem{corollary}{Corollary}[section]
\newtheorem{proposition}{Proposition}[section]
\numberwithin{equation}{section}

\def\ind{1{\hskip -3 pt}\hbox{\textsc{I}}}
\def\n{\noindent}

\def\Ga{\Gamma}

\def\va{\varphi}

\def\o{\omega}
\def\O{\Omega}

\def\bR{\mathbb R}
\def\al{\alpha}

\def\w{\wedge}
\def\cn{\mathbb C^n}
\def\om{\omega}
\def\pa{\partial}

\def\R{\mathbb{R}}

\def\ga{\gamma}
\def\wed{\wedge}
\def\Om{\Omega}

\def\d{\partial}

\begin{document}

\title[Continuous solutions of Dirichlet problem to Hessian type equations]%
{Continuous solutions of Dirichlet problem to Hessian type equations for $(\omega,m)-\beta$-subharmonic functions on a ball in $\mathbb{C}^n$}
\author[L. M. Hai \and N. V. Phu \and T. Tung]%
{Le Mau Hai* \and Nguyen Van Phu** \and Trinh Tung***}
\date{\today,\ Revision}
\thanks{Preprint submitted to \textsc{Mathematica Slovaca}}
\newcommand{\acr}{\newline\indent}

\address{\llap{*\,}Department of Mathematics,\acr
                   Hanoi National University of Education,\acr
                   136-Xuan Thuy Rd., Hanoi,\acr VIETNAM.
                   \acr
                  Orcid: \url{https://orcid.org/0000-0002-2264-8305}}
\email{mauhai@hnue.edu.vn}

\address{\llap{**\,}Department of Mathematics, Faculty of Natural Sciences,\acr
                    Electric Power University,\acr
                    235-Hoang Quoc Viet Rd., Hanoi,\acr
                     VIETNAM.\acr
                    Orcid: \url{https://orcid.org/0000-0002-2851-7250}}
\email{phunv@epu.edu.vn}

\address{\llap{***\,}Department of Mathematics, Faculty of Natural Sciences,\acr
	Electric Power University,\acr
	235-Hoang Quoc Viet Rd., Hanoi,\acr
	VIETNAM.\acr
	Orcid: \url{https://orcid.org/0000-0002-0681-7447}}
\email{tungtrinhvn@gmail.com}

\subjclass[2020]{Primary 32U05; Secondary 32Q15, 32W20} 
\keywords{$m-\beta$-subharmonic functions, $(\omega,m)-\beta$-sh functions, Hermitian forms, Complex Hessian type equation.}

\begin{abstract}
In this paper, we investigate the continuity of solutions to the Dirichlet problem for complex Hessian-type equations associated with $(\omega, m)-\beta$-subharmonic functions on a ball in $\mathbb{C}^n$, where
$
\beta=d d^c\|z\|^2=\frac{i}{2} \sum_{j=1}^n d z_j \wedge d \bar{z}_j
$
is denoted the flat metric on $\mathbb{C}^n$.
\end{abstract}

\maketitle
\section{Introduction}

\n \textbf{Background.} Throughout this paper, we consider the problem on a ball, which we denote by $\Omega$. Moreover, on $\mathbb{C}^n$, we use the flat Hermitian metric induced by the canonical K\"ahler form $\beta=\frac{\sqrt{-1}}{2}\sum\limits_{j=1}^n dz_j\wed d\bar{z}_j$. Let $1\leq m\leq n$ be an integer. Set 
$$\Gamma_m(\beta)=\big\{\gamma: \gamma \text{ is a real $(1,1)$-form},\gamma^k\wed \beta^{n-k}(z)>0\big\},\ k = \overline{1,m},$$

\n for all $z\in\Omega$. Assume that $\om\in \Gamma_m(\beta)$ is a smooth real $(1,1)$-form on $\Omega$. Let $u\in C^2(\Om,\bR)$ be given. The complex Hessian operator related to $\om$ acting on  $u$ is defined by 
$$
H_{m,\o}(u) = (\o+dd^c u)^m \wedge \beta^{n-m}.
$$

For $\om = 0,$    B{\l}ocki in \cite{Bl05}  defined the Hessian operator acting on bounded $m$-subharmonic functions (not necessarily smooth).
In this case, the Hessian operator $H_{m,0}(\bullet)$ is a positive Radon measure which is stable under monotone sequences, and the homogeneous Dirichlet problem is solvable on a ball in $\cn$. To make it convenient for presentation, from now on we use the symbol $H_m(\bullet)$ instead of $H_{m,0}(\bullet)$.

According to \cite{GN18,KN23c}, a function $u:\Om \to [-\infty,+\infty)$ is called $(\o,m)-\beta$-subharmonic (abbreviation is $(\o,m)-\beta$-sh) if $u$ is upper semi-continuous on $\O$ and $u\in L^1_{\rm loc}(\Om,\beta^n)$ and for any collection $\ga_1,...,\ga_{m-1} \in \Ga_m(\beta)$, 
\begin{equation}\label{1.1}
(\o+dd^c u )\wed \ga_1 \wed \cdots \wed \ga_{m-1} \wed \beta^{n-m}\geq 0
\end{equation} in the sense of currents.
The cone of $(\o,m)-\beta$-sh (resp. negative) functions on $\Om$ is denoted by $SH_{ \om,m} (\Om$) (resp. $SH_{ \om,m}^{-} (\Om)$). In the case when $\o=0,$ to make it convenient for readers, from now on we use symbol $SH_{m} (\Om$) (resp. $SH_{m}^{-} (\Om)$) instead of $SH_{ 0,m} (\Om$) (resp. $SH_{ 0,m}^{-} (\Om)$).\\
For a $C^2$ function $u$, the inequality \eqref{1.1} is equivalent to inequalities
$$
(\o+dd^cu)^k\wed\beta^{n-k}(z)\geq 0\,\,\text{for}\,\,k=1,\ldots,m,
$$
\n for all $z\in\O$.
Let $\rho$ be a strictly plurisubharmonic function on $\O$ such that $dd^c\rho\geq \o$ on $\Omega$ and $u\in SH_{ \om,m} (\Om)$. Then by Definition \eqref{1.1} we have $u+\rho$ is a $m$-sh function on $\O$. We write \begin{equation*}
\tau=dd^c\rho-\o,
\end{equation*} which is a smooth $(1,1)$-form. It follows that $\omega+dd^cu=dd^c(u+\rho)-\tau.$ We define $$H_{m,\om}(u) =(\o+dd^cu)^m\w\beta^{n-m}:=\sum_{k=0}^m\binom{m}{k}(-1)^{m-k}[dd^c(u+\rho)^k]\w\tau^{m-k}\w\beta^{n-m}.$$

\n Following the traditional inductive method as in \cite{BT1}, the Hessian operator $H_{m,\omega} (u)$ can be defined
over the class of locally bounded $(\o,m)-\beta$-sh functions $u$ as  positive Radon measure which puts no mass on $m$-polar sets. See Section 9 in \cite{KN23c} and \cite{GN18} for details. 
Moreover, in \cite{GN18}, Dongwei Gu and Ngoc Cuong Nguyen developed a pluripotential theory parallel to the classical framework established in \cite{BT1}. Within this setting, several fundamental properties such as the continuity under monotone sequences of the Hessian operator, quasi-continuity, and the comparison principle remain valid. It is worth emphasizing that Remark 3.19(b) in \cite{GN18} asserts that any bounded \((\omega,m)\)-\(\beta\)-subharmonic function can be approximated by a decreasing sequence of smooth $(\o,m)-\beta$-sh functions.

\n 
From now on, unless otherwise specified, we always assume that 
$\mu$ is a positive Radon measure on 
$\Om ,\phi:\partial \O \to \mathbb{R}$ is continuous function 
and $F(t,z): \R\times \Om \to [0,+\infty)$ is a given function.
The Hessian type equation is to seek  $u\in SH_{\omega,m}(\O)$ such that 

\begin{equation}\label{eq1.5}
\begin{cases}
u\in SH_{\om,m}(\Om)\cap L^{\infty}(\overline{\Om})\\
H_{m,\o}(u)=F(u,z)d\mu\\
\lim\limits_{\Omega\ni z\to x}u(z)=\phi (x), \forall x\in\pa \Om. 
\end{cases}
\end{equation}
In the case when \( m = n \) and \( \omega = 0 \), this equation is referred to as a Monge-Amp\`ere type equation. Monge-Amp\`ere type equations have attracted the interest of many mathematicians.
When \( d\mu = dV_{2n} \), where \( dV_{2n} \) denotes the Lebesgue measure on \( \mathbb{C}^n \), Bedford and Taylor \cite{BT79} established the existence and uniqueness of a continuous solution to the problem \eqref{eq1.5} under the assumption that \( F(t, z) \in C^0(\mathbb{R} \times \overline{\Omega}) \), and that \( F^{1/n} \) is convex and non-decreasing in \( t \).
Subsequently, still in the setting \( m = n \) and \( \omega = 0 \), Cegrell \cite{Ce84} extended this result by proving the existence of solutions when \( F(t,z) \) is a bounded function that is continuous with respect to the first variable for each fixed \( z \in \Omega \).
Furthermore, in the case where the measure \( d\mu = (dd^c v)^n \) with \( v \in \mathrm{PSH}(\Omega) \cap L^\infty(\Omega) \) satisfying \( \lim\limits_{z \to x} v(z) = \phi(x) \) for all \( x \in \partial \Omega \), and where \( F(t, z) \) is a bounded function that is non-decreasing and continuous in the first variable and \( d\mu \)-measurable in the second one, Ko{\l}odziej \cite{K00} proved the existence and uniqueness of solutions to the Monge-Amp\`ere type equation.
More recently, when \( \Omega = M \) is a compact Hermitian manifold with boundary $\pa M $, $\pa M$ is not a empty set and $\om$ is a Hermitian metric on $M$, the problem \eqref{eq1.5} has been discussed in \cite[Theorem 2.3]{KN23b}. 
In this work, Ko{\l}odziej and Cuong established the existence of bounded solutions of Equation \eqref{eq1.5} in the case $m=n$, i.e, the existence of bounded solutions of Monge-Amp\`ere type equation on a compact Hermitian manifold  $\overline{M}= M\cup\pa M$. Note that, in this case, they put some assumptions for the function $F(t,z)$. Namely, the function $F(t,z)$ is bounded and continuous and non-decreasing in the first variable.  Next, in \cite{KN23} Ko{\l}odziej and Cuong solved the Dirichlet problem for the Monge-Amp\`ere equation on Hermitian manifolds with boundary $M$. They investigate the following problem.
\begin{equation}\label{1.6}
\begin{cases}
u\in PSH(M,\om)\cap L^{\infty}(\overline{M})\\
(\om + dd^c u)^n=\mu\\
\lim\limits_{M\ni z\to x}u(z)=\varphi (x), \forall x\in\pa M. 
\end{cases}
\end{equation}
Here $(\overline{M},\om)$ is a $C^{\infty}$ smooth compact Hermitian manifold of dimension $n$ with non-empty boundary $\pa M$ and $\overline{M}= M\cup\pa M$. The solution of Equation \eqref{1.6} is sought in the class of $\omega$-plurisubharmonic functions on $M$. This constitutes the Dirichlet problem for the complex Monge-Amp\`ere equation on a Hermitian manifold with boundary.
Theorem~\textbf{1.2} in the aforementioned paper confirms that if there exists a subsolution $\underline{u}$ to Equation \eqref{1.6}, then  Equation \eqref{1.6} has a solution in the class of $\om$-plurisubharmonic functions on $M$. It is also worth noting that in the above-mentioned paper, the authors investigated the continuity of solutions of Equation \eqref{1.6}. According to Corollary {\bf 1.3} in that paper, if the measure $\mu$ belongs to $\mathcal{F}(M, h)$ for some admissible function $h$, then the solution to Equation \eqref{1.6} is continuous on $\overline{M}$.

Being directly influenced by the above is also the driving motivation for this work. In this paper, we want to extend the results of the above authors from Monge-Amp\`ere equations to Hessian-type equations. First, we establish the existence of bounded solutions for Hessian type equations \eqref{eq1.5} in the class $SH_{\om,m}(\O)$ of $(\om,m)-\beta$-sh functions introduced and investigated in \cite{KN23c}. The method used to prove the existence of  bounded solutions to the Hessian type equation \eqref{eq1.5} in the class $SH_{\omega,m}(\Omega)$ is based on fixed point theory. However, unlike the similar technique used in \cite{KN23} for Monge-Amp\`ere type equations in the class of $\omega$-plurisubharmonic functions, here we do not require the function $F(t,z)$ to be bounded and non-decreasing in the first variable $t$. Next, we aim to investigate under what conditions the solution to equation \eqref{eq1.5} becomes continuous, as in equation \eqref{1.7} below. To obtain the desired result, we impose an additional condition on the measure  $\mu$, namely that $\mu$ satisfies the $\gamma$-Dini-diffuse condition. Utilizing some recent results, including  Lemma 14 in \cite{HQ23}  and Lemma 1.6 in \cite{ChaZe24}, allows us to establish the continuity of the solution to equation \eqref{eq1.5} in the case $\Omega$ is a ball in $\cn$. Namely, we have the following result. Under the same hypotheses for solving equation \eqref{eq1.5} and adding the condition that the measure $\mu$ is $\gamma$-Dini-diffuse,  then the following equation has a unique continuous solution.  

\begin{equation}\label{1.7}
\begin{cases}
u\in SH_{\om,m}(\Om)\cap C^{0}(\overline{\Om})\\
H_{m,\o}(u)=F(u,z)d\mu\\
\lim\limits_{\Omega\ni z\to x}u(z)=\phi (x), \forall x\in\pa \Om. 
\end{cases}
\end{equation}

\n Our main result is stated in the following theorem:
\begin{theorem}[\textsc{Main Theorem}] \label{th1.1}
	Assume that the following conditions are satisfied:
	
	\n (a) $F(t,z)$ is semi-upper continuous on $\mathbb{R}\times\O$ and $t\mapsto F(t,z)$ is continuous.
	
	\n (b) There exist a function $0<G\in L^1_{loc}(\O,\mu)$ and $v\in SH_{m}^-(\Om)\cap L^{\infty}(\Om)$ satisfying $F(t,z)\leq G(z)$ for all $(t,z)\in\mathbb{R}\times\O$ and 
	$$\lim\limits_{\Om \ni z\to x}v(z)=0\,\forall x\in\pa \Om \ \text{and}\ G\mu\leq H_{m}(v);$$ 
	
	\n Then the following assertions hold:
	
	\n (i) The Dirichlet problem \eqref{eq1.5} has a  solution.\\
	Moreover, if $ F(t,z)$ is a non-decreasing function with respect to the first variable for every 
	$z \in \O \setminus X$ where $X$ is a Borel set with $C_m (X)=0,$ then the Dirichlet problem \eqref{eq1.5} has a unique solution.
	
	\n (ii) If, in addition to the above conditions, we assume that $F(t,z)$ is bounded on $\mathbb{R}\times\O$ and $\mu$ is {Dini--$\gamma$--diffuse} with respect to $Cap_{\om,m}(\bullet)$ then problem \eqref{1.7} has a unique solution. 
\end{theorem}

\noindent \textbf{Organization of the paper.} The structure of this paper is divided into three sections. In Section \ref{sec2}, following seminal works  \cite{GN18,KN23c}, among others, we recall the basic properties of \( (\omega, m) \)-\( \beta \)-subharmonic functions on \( \Omega \). Of particular importance is the comparison principle for the Hessian operator \( H_{m,\omega}(u) \). We also recall the subsolution theorem, a powerful tool for verifying the existence of solutions to the Dirichlet problem for the Hessian-type equation.
In Section \ref{sec3}, we supply in detail the proofs of our main result.
\vskip 0.3cm

\n 
{\bf Acknowledgments.} 
The first and second named authors are supported by Grant number B2025-CTT-10 from the Ministry of Education and Training, Vietnam.
This work was written during our visit to the Vietnam Institute for Advanced Study in Mathematics (VIASM) in the Spring of 2025. We also thank VIASM for its hospitality.

\section{Preliminaries and auxiliary results}\label{sec2}

Following Proposition 2.6 in  \cite{GN18}, we include below some basic properties of $(\o,m)-\beta$-sh functions. 

\begin{proposition} \label{prop:closure-max}
	\begin{enumerate}
		\item[(a)]
		$SH_{\o,n}(\O)\subset SH_{\o,n-1}(\O)\subset\cdots\subset SH_{\o,1}(\O)$
		\item[(b)]	If $u_1 \geq u_2 \geq  \cdots$ is a decreasing sequence of $(\o,m)-\beta$-sh functions, then $u := \lim_{j\to \infty} u_j$ is either $(\o,m)-\beta$-sh or $\equiv -\infty$.
		\item[(c)] 
		If $u, v$ belong to $SH_{ \om,m}(\Omega)$, then so does $\max\{u,v\}$.
		\item[(d)] 
		If $u, v$ belong to $SH_{ \om, m}(\Omega)$ and satisfies $u \le v$ a.e. with respect to Lebesgue measure then $u \le v$ on $\Om.$
		\item[(e)] Let $u_{\alpha}\in SH_{\om,m}(\O)$ be a family that locally uniformly bounded above. Put $u(z)=\sup\limits_{\alpha}u_{\alpha}(z)$. Then, the upper semi-continuous regularization $u^{*}$ is a $(\om,m)$-$\beta$-subharmonic function on $\O$.
	\end{enumerate}
\end{proposition}
\n
Notice that $(d)$ follows from Lemma 9.6 and Definition 2.4 in \cite{GN18}  where $u+\rho$ and $v+\rho$ are viewed as $\al-$subharmonic function with any $(1,1)$ form $\al$ such that 
$
\al^{n-1} =\ga_1\w\cdots\w \ga_{m-1} \wed \beta^{n-m},$ where $\ga_1,\ldots,\ga_{m-1}$ is a certain form belonging to $\Ga_m (\beta)$.

\begin{definition}\label{dn22}
	For a Borel set $E\subset \Om,$ we set
	\begin{align*}
	Cap_{\om,m}(E,\O)&= Cap_{\om,m}(E)\\
	&:=\sup\bigg\{\int_{E}(\o+dd^cv)^m \wedge \beta^{n-m}: v\in SH_{\om,m}(\Om), 0\leq v\leq 1\bigg\},
	\end{align*}
	$$C_m(E):=\sup\bigg\{\int_{E}(dd^cv)^m \wedge \beta^{n-m}: v\in SH_{m}(\Om), 0\leq v\leq 1\bigg\}.$$
\end{definition}
\n According to Lemma 3.5 in \cite{GN18}, there exists a constant $C$ depending on $\o$ such that \begin{equation}\label{e2.1}\frac{1}{C}Cap_{\om,m}(E)\leq C_m(E)\leq C. Cap_{\om,m}(E).\end{equation}
\begin{remark}\label{remark}{\rm By Corollary {\bf 3.2} in \cite{SA12} we note that if $C_m(E)=0$ then $E$ is a $m$-polar subset in $\O$. Hence, by \eqref{e2.1} we get that if $Cap_{\om,m}(E)=0$ then it follows that $E$ is a $m$-polar subset in $\O$.}
\end{remark}
\n Similar to Definition {\bf 4.2} in \cite{KN23c},  we have the following definition.
\begin{definition}
	A sequence of Borel functions $u_j$ in $\Omega$ is said to converge in $(\o,m)$-capacity (or in $Cap(.))$ to $u$ if for any $\delta>0$ and $K\Subset \Omega$ we have 
	$$\lim\limits_{j\to\infty}Cap(K\cap|u_j-u|\geq\delta)=0.$$
\end{definition}
\begin{remark}\label{monocap}
	\textup{ Repeating argument as in Corollary 4.11 in \cite{KN23c} and inequality \eqref{e2.1}, monotone convergence of locally uniformly bounded sequences of
		$(\o,m)-\beta$-sh functions is convergence in $(\o,m)$-capacity.}
\end{remark}

\n A major tool in pluripotential theory is the comparison principle.
We recall a version of this result for bounded $(\o,m)-\beta$-sh. functions (see Corollary 3.11 in \cite{GN18}).

\begin{theorem}\label{comparison}
	Let $u,v$ be bounded $(\o,m)-\beta$-sh functions in  $\overline{\Omega}$ such that $\liminf\limits_{z\to\partial\Omega}(u-v)(z)\geq 0.$ Assume that $H_{m,\o}(v)\geq H_{m,\o}(u)$ in $\Omega.$ Then  $u\geq v$ on $\Omega.$
\end{theorem}
\n 
Now, we will prove the following result. Note that this result was proved in the case of $m-$ subharmonic functions (see Lemma 4.5 in \cite{PD24}).
\begin{lemma}\label{bd1}
	Assume that $\mu$ vanishes on $m-$polar sets of $\Om$ and $\mu(\Omega) < \infty.$ Let $\{u_{j}\}\in SH_{\om,m}^{-} (\Omega)$ be a sequence satisfying the following conditions:
	
	\n 
	(i) $\sup\limits_{j \ge 1} \int\limits_{\Om} -u_jd\mu <\infty;$
	
	\n 
	(ii) $u_j \to u \in SH_{\om,m}^{-} (\Om)$ a.e. $dV_{2n}.$
	
	Then we have 
	$$\lim_{j \to \infty} \int\limits_{\Om} \vert u_j- u \vert d\mu=0.$$
	In particular $u_j \to u$ a.e. $d\mu$ on $\Om.$
\end{lemma}		
\begin{proof}
	Let $\rho\in C^2(\overline{\O})\cap PSH(\O)$  such that $dd^c\rho\geq \o$. Then $u_j+\rho$ and $u+\rho$ are $m$-subharmonic. Applying Lemma 4.5 in \cite{PD24} with functions $u_j+\rho$ and $u+\rho$ we get the desired  conclusion.
\end{proof}

\n
Now we formulate the following subsolution Theorem  which plays a important role in our work (see Lemma 9.3 in \cite{KN23c}).
\begin{theorem}\label{sub}
	Assume that there exists $v \in SH_{m} (\Om) \cap L^\infty (\Om)$ satisfying 
	$$\label{eq:bounded-subsol}
	H_{m}(v) \geq \mu, \quad \lim_{z\to \d \Om}  v(z) = 0.
	$$	
	Then, there exists a unique bounded $(\o,m)-\beta$-sh function $u$ solving  
	$$\lim_{z\to x} u(z) = \phi (x), \forall  x\in\d\Om,
	H_{m,\o}(u) = \mu\ \text{on}\  \Om.
	$$
\end{theorem}

\n We also need the following result.
\begin{proposition}\label{md1}
	Let $\mu $ be a positive Radon measure on $\O$ and 
	$F(t,z): \R\times \Om \to [0,+\infty)$ be a upper semi-continuous function,  $t\mapsto F(t,z)$ be continuous in the first variable $t$. Moreover, assume that there exists a function $G\in L^1_{loc}(\O,\mu)$ such that
	$$F(t,z)\leq G(z), \\ \text{for all $(t,z)\in \mathbb{R}\times\O$}.$$
	Then 
	$F(u,z)d\mu$ is a positive Radon measure on $\O,$ where $u$ is upper semi-continuous function on $\O$.
\end{proposition}
\begin{proof}
	First, we show that $F(u,z)$ is $\mu-$ measurable on $\O$.	
	Let $u_k$ be a sequence of continuous functions on $\O$ that decreases to $u$ pointwise on $\O.$
	It follows from the assumption $F(t,z)$ is continuous in the first variable that $F(u_k, z)$ converges pointwise to $F(u,z)$ for all $z \in \Om.$ 
	Now we claim that $F(u_k, z)$ is upper semicontinuous on $\O.$ Indeed, fix a sequence $\O \ni \{z_j\} \to z^* \in \O.$
	Then $u_k (z_j) \to u_k (z^*)$ as $j \to \infty.$ Hence $(u_k (z_j), z_j) \to (u_k (z^*), z^*)$ as $j \to \infty$. It follows from the  hypothesis $F$ is upper semi-continuous on $\mathbb{R}\times\Om$  that
	$$\limsup_{j \to \infty} F(u_k (z_j), z_j) \le F(u_k (z^*), z^*).$$
	Hence $F(u_k,z)$ is non-negative upper semi-continuous on $\O$ as claimed. This yields that  $F(u_k,z)$ is $\mu-$ measurable on $\Om$ and, consequently, $F(u,z)$ is $\mu-$ measurable on $\O.$ 
	To finish off, we should check $F(u,z)d\mu$ is a continuous linear functional on the space $C^{\infty}_0(\Om)$ of smooth functions with compact support on $\Om$. It is clear the mapping $\va\mapsto \int\limits_{\O}\va F(u,z)d\mu$ is a linear functional on $C^{\infty}_0(\Om)$. Next, we need to check that it is continuous on $C^{\infty}_0(\Om)$ with the norm $\|\va\|_{\infty}=\sup\limits_{x\in\O}|\va(x)|$ for $\va\in C^{\infty}_0(\Om)$. Let $\va \ge 0$ be a smooth function with compact support in $\O.$ Then we have
	\begin{align*}
	0\leq\int\limits_{\O}\va F(u,z)d\mu&\leq \int\limits_{\O}\va G(z) d\mu\\
	&=\int\limits_{{\rm supp \va}}\va G(z)d\mu\leq \|\va\|_{\infty}\int\limits_{{\rm supp \va}}G(z) d\mu<+\infty.
	\end{align*}
	
	\n	Therefore, by the Riesz representation theorem, $F(u,z)d\mu$ can be identified with a positive Borel measure on $\Om$. Now we prove $F(u,z) d\mu$ is a positive Radon measure. Let $K\Subset\O$. Take an open subset $U\Subset\O$ with $K\Subset U\Subset\O$. Choose $\va\in C^{\infty}_0(\Om)$ with ${\rm supp}\va\subset U$ and $\va=1$ on $K$, $0\leq \va\leq 1$. Then we have
	$$\int\limits_{K} F(u,z) d\mu\leq \int\limits_{U}\va F(u,z)d\mu<+\infty,$$
	
	\n and the required conclusion follows.
\end{proof}

\n Finally, we deal with the property {Dini-$\gamma$-diffuse} of a Borel measure $\mu$ on $\O$. This property originates from the concept of diffuse Borel measure in the paper of  Charabati and Zeriahi in \cite{ChaZe24}. It is also used effectively in \cite{HQ23} to establish the continuity of solutions of degenerate complex Hessian equations on Hermitian manifolds. Now we recall this notion from \cite{HQ23}. Readers can explore this concept in more depth through  \cite{ChaZe24} and \cite{HQ23}.    
\begin{definition}\label{dn25}
	\n Let $\O$ be a ball in $\cn$ and $\mu$ be a positive Borel measure on $\O$. Let $\omega\in \Gamma_m(\beta)$ be a smooth real $(1,1)$-form. Measure $\mu$ is said to be 	{Dini-$\gamma$-diffuse} w.r.t. $Cap_{\om,m}(\bullet)$ on $\O$ if there
	exists a function $\gamma : [0, +\infty) \longrightarrow [0, +\infty)$ satisfying the following conditions
	
	\n (i) $\mu(E) \leq Cap_{\om,m}(E) \cdot \gamma [Cap_{\om,m}(E)]$, for every Borel subset $E \subset \O$. 
	
	\n (ii) $\gamma$ is non-decreasing on $(0, b_{\gamma})$, where $0 < b_{\gamma} \leq +\infty$.
	
	\n (iii)$\int\limits_{0}^{b_{\gamma}}\frac{\gamma^{\frac{1}{m}}(t)}{t} dt<+\infty,$ (this condition is called the Dini type condition for $\gamma$)
	where $Cap_{\om,m}(E)$ is a $(\om,m)-$capacity of a Borel subset $E\subset\O$ given by Definition \ref{dn22}.
\end{definition}

\section{Continuous solutions of the Dirichlet problem to  Hessian type equation for $(\om,m)-\beta$-sh functions on a ball in $\cn$}\label{sec3}

The section is devoted to the proof of Theorem \ref{th1.1}. We split the proof into two steps. Firstly, under the assumptions (a) and (b), we will prove the existence of solutions of equation \eqref{eq1.5}. Here, as a traditional approach, we exploit the theory of fixed points. However, in contrast to some previous works, we do not assume that the function $t\mapsto F(t, z)$ is non-decreasing with respect to $t$. Now, we prove the comparison principle in the class $SH_{\om,m}(\Om)\cap L^{\infty}(\Om)$ which is similar to Proposition 2.2 in \cite{KN23b} for $\om$-plurisubharmonic functions. Namely, we need the following result.

\begin{theorem}\label{thm 5.3}
	If $u, v \in SH_{\om,m}(\Om)\cap L^{\infty}(\Om)$, then 
	$$ \ind_{\{u>v\}} ( \omega + dd^c u)^m\w\beta^{n-m} = \ind_{\{u>v\}} ( \omega + dd^c \max(u,v))^m\w\beta^{n-m}. $$
\end{theorem}
\begin{proof}
	Let $\rho$ be a smooth strictly  psh function on $\O$ such that $dd^c\rho\geq\om$ on $\O$.  Write 
	$$H_{m,\o}(u)  = \sum_{k=0}^n \binom{m}{k} (-1)^{m-k} ( dd^c (u+\rho))^k \wedge (dd^c \rho - \omega)^{m-k}\w\beta^{n-m},$$
	$$H_{m,\o}( \max(u,v))  = \sum_{k=0}^n \binom{m}{k} (-1)^{m-k} ( dd^c ( \max(u,v)+\rho))^k \wedge (dd^c \rho - \omega)^{m-k}\w\beta^{n-m},$$
	
	\n	Note that $\omega$ is a smooth real $(1,1)$-form, we deduce $\tau^{m-k}=(dd^c \rho - \omega)^{m-k}$ is a smooth real $(m-k,m-k)$-form. Hence, according to Proposition 4.1 in \cite{Sa25} we can write
	\begin{equation}\label{tau1}
	\tau^{m-k}= (dd^c \rho - \omega)^{m-k} = \sum\limits_{j\in J} f_j T_j, 
	\end{equation} 
	where $J$ is a finite set, $(f_j)_{j\in J}$ are smooth functions with complex values  and $T_j=dd^cu^j_1\w\cdots\w dd^cu^j_{m-k}$  where, for every $j\in J$ and every $i=1,\cdots,m-k, u_i^j$ is a smooth negative plurisubharmonic function defined in a neighborhood of $\overline{\Omega}$. Note that, $\tau^{m-k}$ and $T_j$ are smooth real forms. Hence, we have
	\begin{equation}\label{tau2}
	\tau^{m-k}= \overline{\tau^{m-k}}=\overline{\sum\limits_{j\in J} f_j T_j}=\sum\limits_{j\in J} \overline{f_j}.   \overline{T_j}= \sum\limits_{j\in J} \overline{f_j} T_j
	\end{equation}
Combining the equations  \eqref{tau1} and \eqref{tau2}, we infer that
$\tau^{m-k}=\sum\limits_{j\in J} \frac{f_j + \overline{f_j}}{2}T_j = \sum\limits_{j\in J} g_j T_j ,$
where $(g_j)_{j\in J}$ are smooth functions with real value.
	 By linearity, it suffices to prove that
	$$\ind_{\{ u > v\}} ( dd^c (u+\rho))^k \wedge T_j\w\beta^{n-m}= \ind_{\{ u > v\}} (  dd^c \max(u+\rho,v+\rho))^k \wedge T_j \w\beta^{n-m}.  $$
	By Lemma 4 in \cite{DE16} the above equality holds. The proof is complete.
\end{proof}

The following corollary is an extension of a result due to Demailly (see \cite[Proposition 6.11]{Dem89} in the case when $m=n$ and $\omega=0$).
\begin{corollary}\label{cor 5.5}
	Let $u, v \in SH_{\om,m}(\Om)\cap L^{\infty}(\Om)$. Then we have 
	$$ ( \omega + dd^c \max(u,v))^m\w\beta^{n-m} \geq \ind_{\{u> v \}} ( \omega + dd^c u)^m\w\beta^{n-m} + \ind_{\{ u \leq v \}} ( \omega + dd^c v)^m\w\beta^{n-m}. $$
\end{corollary}
\begin{proof}
	By Theorem \ref{thm 5.3}	we have 
	\begin{align*}
	&( \omega + dd^c \max(u,v))^m\w\beta^{n-m}\\
	& \geq  \ind_{\{u>v\}} ( \omega + dd^c \max(u,v))^m\w\beta^{n-m} +  \ind_{\{u < v\}} ( \omega + dd^c \max(u,v))^m\w\beta^{n-m}\\
	& \geq  \ind_{\{u>v\}} ( \omega + dd^c u)^m\w\beta^{n-m} +  \ind_{\{u < v\}} ( \omega + dd^c v)^m\w\beta^{n-m}. 
	\end{align*}
	If $[( \omega + dd^c v)^m\w\beta^{n-m}] (\{u = v\}) = 0$, then the result follows. 
	In the case when $$[( \omega + dd^c v)^m\w\beta^{n-m}] (\{u = v\}) \ne 0,$$	since $(\omega + dd^c v)^m\w\beta^{n-m}$ 
	vanishes on $m$-polar sets,	the proof in \cite{HP17}(see Proposition 5.2) shows that 
	$$ [(\omega + dd^c v)^m\w\beta^{n-m}](\{ u = v + t \}) = 0, \; \;  \forall t \in \mathbb{R}\setminus I_{\mu}, $$
	where $I_{\mu}$ is at most countable. Take $\varepsilon_j \in \mathbb{R}\setminus I_{\mu}$ and  $\varepsilon_j \searrow 0$.   We have 
	
	$$[(\omega + dd^c v)^m\w\beta^{n-m}](\{ u = v + \varepsilon_j \}) = 0.$$
	This implies that	
	\begin{align*} ( \omega + dd^c \max(u,v+ \varepsilon_j))^m\w\beta^{n-m} 
	\geq \ind_{\{u> v+ \varepsilon_j \}} ( \omega + dd^c u)^m\w\beta^{n-m} + \ind_{\{ u \leq v+ \varepsilon_j \}} ( \omega + dd^c v)^m\w\beta^{n-m}. 
	\end{align*}
	
	\noindent	Let $\varepsilon_j \rightarrow 0$, according to Theorem 3.4 in \cite{GN18} and the Lebesgue Monotone Convergence Theorem, we get the desired conclusion.
\end{proof}

We need a version of the comparision principle. 
See  Proposition 2.2 in \cite{KN23b} for an analogous result for $\om$-plurisubharmonic  functions.

\begin{proposition}\label{md4}
	Let $\nu \ge \mu$ be positive Radon measures on $\Om.$ 
	Assume that	$ t \mapsto F(t,z)$ is a non-decreasing function in $t$ 
	for all $z \in \Om \setminus Z,$ where $Z \subset \Om$ is a Borel set with $C_m (Z)=0.$
	Let
	$u,v\in SH_{\om,m}(\Om)\cap L^{\infty}(\Om)$ be functions satisfying the following conditions: 
	
	\n 
	(i)
	$\liminf\limits_{z\to\pa \Om}(u-v)(z)\geq 0;$
	
	\n
	(ii) $H_{m,\o}(u)=F(u,z)\mu, H_{m,\o}(v)=\tilde F(v,z)\nu,$ where $\tilde F \ge F$ is a
	measurable function on $\Om.$
	
	\n Then $u\geq v$ on $\Om.$
\end{proposition}

\begin{proof}

	It follows from Corollary \ref{cor 5.5} that 
	\begin{align*}
	(\omega + dd^c \max(u,v))^m\w\beta^{n-m}
	&\geq \ind_{\{u> v \}} ( \omega + dd^c u)^m\w\beta^{n-m} + \ind_{\{ u \leq v \}} ( \omega + dd^c v)^m\w\beta^{n-m} \\
	&= \ind_{\{u> v \}} F(u,.) d\mu  + \ind_{\{ u \leq v \}} \tilde{F}(v,.) d\nu \\
	&\geq \ind_{\{u> v \}} F(u,.) d\mu  + \ind_{\{ u \leq v \}} F(v,.) d\mu \\
	&= F(\max(u,v),.) d\mu \\
	&\geq F(u,.) d\mu = (\omega + dd^c u)^m\w\beta^{n-m},
	\end{align*}
	where the last inequality follows from the fact that the function $F$ is non-deceasing in the first variable. According to 
	Theorem \ref{comparison}, we infer that  $u \geq v$. The proof is complete.

\end{proof}
\n	Now using the hypotheses (a) and (b) of Theorem \ref{th1.1} we show the existence of bounded $(\om,m)-\beta$-sh solutions of the equation \eqref{eq1.5}. 
\subsection{Proof of the existence of solutions of the equation \eqref{eq1.5}}
It follows from the hypotheses (b) that
 $\mu$ puts no mass on $m-$polar subsets of $\Om.$ 
Next, by Theorem {\bf 3.15} in \cite{GN18} there exists $h\in SH_{\om,m}(\O)\cap C^0(\overline{\Om})$ such that
\begin{equation*} \label{eqh}
H_{m,\o}(h)=0,\, h=\phi \ \text{on}\ \pa \Om.
\end{equation*}

\n Note that $\lim\limits_{z\to x}( v+h)(z)=\phi(x)$. On the other hand, we have $H_{m,\om}(v+h)\geq 0= H_{m,\om}(h)$. So using Theorem \ref{comparison} we obtain $v+h \le h$ on $\Om.$ Set
$$\mathcal{A}=\{\chi\in SH_{\om, m}(\O)\cap L^{\infty}(\bar{\O}): v+h\leq \chi\leq h\}.$$

\n Note that $h\in\mathcal{A}$ then $\mathcal{A}\ne\emptyset$. It is easy to see that $\mathcal{A}$ is a convex subset in $SH_{\om,m}(\O)$. Moreover, $\mathcal{A}$ is a compact subset in $L^1(\O, d\mu)$. Indeed, assume that $\{\chi_j\}_{j\geq 1}$ is a sequence in $\mathcal{A}$. Then for all $j\geq 1$ we have
\begin{equation}\label{eq3.2}
|\chi_j|\leq \max\Bigl(|h|, |v+h|\Bigl)\leq |h|+|v|+|h|.
\end{equation}
Hence, $\int\limits_{\O}|\chi_j|dV_{2n}\leq \int\limits_{\O}|h|+ \int\limits_{\O}(|v|+ |h|)<+\infty,$ for all $j\geq 1$ because $h,v\in L^{\infty}(\bar{\O})$. It follows that $\{\chi_j\}_{j\geq 1}$ is uniformly bounded in $L^1(\O,dV_{2n})$. By Proposition {\bf 9.12} in \cite{GN18} we can take a subsequence of the sequence $\{\chi_j\}_{j\geq 1}$ which we also still denote it by $\{\chi_j\}$ such that $\chi_j\longrightarrow \chi\in SH_{\om,m}(\O)$ in $L^1(\O, dV_{2n})$. On the other hand, by \eqref{eq3.2} it follows that
$$\sup\limits_{j \ge 1}\int\limits_{0}-\chi_j d\mu<\infty,$$

\n by Lemma \ref{bd1}, we achieve that $\chi_j$ is convergent to $\chi$ in $L^1(\O,d\mu)$ as we wanted.\\
Now for $\chi\in \mathcal{A}$ we have
$$F(\chi,z)d\mu\leq G(z) d\mu\leq H_{m}(v),$$

\n by Theorem \ref{sub} there exists unique $g\in SH_{\om,m}(\O)\cap L^{\infty}(\O)$, $\lim\limits_{z\to x} g(z)=\phi(x)$ such that 
$$H_{m,\om}(g)= F(\chi,z)d\mu.$$  

\n Since $\lim\limits_{z\to x} g(z)=\phi(x)$ and $\phi\in C^0(\partial{\O})$ then $g\in L^{\infty}(\bar{\O})$.\\
Moreover, we also have
$$H_{m,\omega}(v+h)\geq H_{m}(v)\geq H_{m,\om}(g)\geq 0= H_{m,\om}(h)$$
By Theorem \ref{comparison} we get that $v+h\leq g\leq h$ on $\O$. Thus it follows $g\in\mathcal{A}$. Now we define a map $T:\mathcal{A} \longrightarrow\mathcal{A}$ by putting $\mathcal{A}\ni \chi\mapsto T(\chi)=g\in\mathcal{A}$. We will prove that $T$ is continuous. Let $\{u_j\}\subset \mathcal{A}$ be a sequence such that $u_j\to u\in \mathcal{A}$ in $L^1(\Om,d\mu).$ Hence, we can assume that $u_j \to u$ almost everywhere with respect $d\mu$. Set $g=T(u)$ and $g_j=T(u_j).$ 
Now for $z \in \Omega,$ we define  the following sequences of non-negative uniformly bounded measurable functions
$$\psi^1_j (z):= \inf_{k \ge j} F(u_k (z),z),
\psi^2_j (z):= \sup_{k \ge j} F(u_k (z),z).$$
Then from the definitions of $\psi^1_j (z)$ and $\psi^2_j (z)$ we deduce that \\
\n 
(i) $0 \le \psi^1_j (z) \le  F(u_j (z),z) \le \psi^2_j (z) \le G(z)$ for $j \ge 1$ and $z\in\O$.

\n Furthermore, because $F(t,z)$ is continuous function in the first variable, we have:\\
\n 
(ii) $\lim\limits_{j \to \infty} \psi^1_j (z)=\lim\limits_{j \to \infty} \psi^2_j (z)= F(u(z),z)$ almost everywhere with respect to the measure $d\mu$.

\n Since $\psi^1_j (z)d\mu\leq F(u_j(z),z)d\mu\leq G(z)d\mu\leq H_m(v)$ by Theorem \ref{sub}, we can find $v^1_j(z)\in SH_{\om,m}(\O)\cap L^{\infty}(\O)$ such that
$$H_{m,\om}(v^1_j)= \psi^1_j (z)d\mu, \lim\limits_{z\to x} v^1_j(z)=\phi(x),$$

\n for all $x\in\pa\O$. Similarly, we can find $v^2_j(z)\in SH_{\om,m}(\O)\cap L^{\infty}(\O)$ such that
$$H_{m,\om}(v^2_j)= \psi^2_j (z)d\mu, \lim\limits_{z\to x} v^2_j(z)=\phi(x),$$

\n for all $x\in\pa\O$.\\
On the other hand, since $H_{m,\om} (h+v)\geq H_m(v)\geq H_{m,\om}(v^1_j)\geq H_{m,\omega}(h)=0$ for all $j\geq 1$, by Theorem \ref{comparison} we get that
$$v+ h\leq v^1_j\leq h.$$
Similarly, we also have $$v+ h\leq v^2_j\leq h,$$

\n for all $j\geq 1$ on $\O$. Moreover, by the definitions of $\psi^1_j$ and $\psi^2_j$ we note that $\psi^1_j$ is an increasing sequence and $\psi^2_j$ is a decreasing sequence. We further have $H_{m,\om}(v^1_j)=\psi^1_j\mu\leq \psi^1_{j+1}\mu=H_{m,\om}(v^1_{j+1})$ then by theorem \ref{comparison} it follows that $v^1_j\geq v^1_{j+1}$ for all $j\geq 1$.  Hence, $v^1_j\searrow v^1$ and by Proposition \ref{prop:closure-max}, $v^1\in SH_{\om, m}(\Omega)$. We also have  $v+h\leq v^1\leq h$. Hence, $\lim\limits_{z\to x} v^1(z)=\phi(x)$ for all $x\in\partial{\O}$. This yields that $v^1\in SH_{\om,m}(\O)\cap L^{\infty}(\bar{\O})$. Hence, we claim that $v^1\in\mathcal{A}$.\\
On the other hand, 
$$H_{m,\om}(v^2_j)=\psi^2_j\mu\geq \psi^2_{j+1}\mu= H_{m,\om}(v^2_{j+1}),$$

\n and $\lim\inf\limits_{z\to \partial{\O}}(v^2_{j+1}- v^2_j)(z)\geq 0$, by Theorem \ref{comparison} we infer that $v^2_j\leq v^2_{j+1}$ on $\O$ for all $j\geq 1$. Thus it follows that $v^2_j \uparrow (v^2)^*$ outside an $m$-polar set with $(v^2)^*\in SH_{\om,m}(\O)$ and $v+h\leq (v^2)^*\leq h$ on $\O$. Hence, we achieve that $(v^2)^*\in\mathcal{A}$. Furthermore, in view of (i) and applying theorem \ref{comparison} we also have
\begin{equation}\label{eq3.5}
v^1_j\geq T(u_j)\geq v^2_j.
\end{equation}
Next we use (ii) to get 
$$H_{m,\o} (v^1_j) \to F (u,z) \mu, H_{m,\o} (v^2_j) \to F(u,z)\mu.$$

\n We also have
$$H_{m,\om}(v^1_j)=\psi^1_j(z)\mu\leq F(u_j(z),z)\mu=H_{m,\om}(g_j)
\leq \psi^2_j\mu= H_{m,\om}(v^2_j)$$

\n and $\lim\limits_{z\to x} v^1_j(z)=\phi(x)=\lim\limits_{z\to x}g_j(x)=\lim\limits_{z\to x}v^2_j(z)$ for all $x\in\partial\Omega$. By applying Theorem \ref{comparison} we infer that
$$v^1_j\geq g_j=T(u_j)\geq v^2_j \ \text{on}\ \O.$$

\n Because $v^1_j\searrow v^1$ as $j\to\infty$ then applying theorem {\bf 3.4} in \cite{GN18} we have $H_{m,\om}(v^1_j)\longrightarrow H_{m,\om}(v^1)$. Similarly, since $v^2_j\nearrow (v^2)^*$, by Remark \ref{monocap}  we get that $H_{m,\om}(v^2_j)\longrightarrow H_{m,\om}((v^2)^*)$. Hence, we have
$$H_{m,\om}(v^1)= H_{m,\om}((v^2)^*)= F(u(z),z)\mu= H_{m,\om}(T(u)).$$

\n However, $\lim\limits_{z\to x}v^1(z)=\phi(x)= \lim\limits_{z\to x}(v^2)^*(z)= \lim\limits_{z\to x}T(u)(z)$ for all $x\in\partial\Omega$ then by Theorem \ref{comparison} we have
$$v^1=(v^2)^*= T(u),$$

\n on $\O$. Now in \eqref{eq3.5}, by letting $j\to \infty$ it follows that 
$$v_1\geq \lim\limits_{j\to\infty}T(u_j)\geq (v^2)^*.$$

\n Therefore, $T(u_j)\longrightarrow T(u)$ outside an $m$-polar subset of $\O$. Since $\mu$ puts no mass on $m$-polar sets, we may apply Lebesgue dominated convergence theorem
to achieve that $T(u_j) \to T(u)$  in $L^1 (\Omega, d\mu)$. Thus $T: \mathcal A \to \mathcal A$ is continuous. 	
Now the Schauder fixed - point theorem implies that there exists $\tilde{u}\in\mathcal{A}$ such that $T(\tilde{u})=\tilde{u}.$
Thus $\tilde{u}$ is a solution of equation \eqref{eq1.5}. 

\n In the case, $ t \mapsto F(t,z)$ is non-decreasing for every 
$z \in \O \setminus Y$ with $C_m (Y)=0,$ then the uniqueness of $\tilde{u}$ follows directly from Proposition \ref{md4} (with $F=\tilde F, \mu=\nu$).	  

\subsection{Proof of the continuity of solutions of the equation \eqref{1.7}} 

\n Let $u\in SH_{\om,m}(\O)\cap L^{\infty}(\overline{\O})$ be a solution of Equation \eqref{eq1.5}. In the final part of this paper we will prove the continuity of $u$ under the assumption that the measure $\mu$ is {$\gamma$-Dini-diffuse} and the function $F(t,z)$ satisfying the condition $$0 < F(t,z)\leq C \quad  \text{for all} (t,z)\in\mathbb{R}\times\O.$$ Without loss of generality, we may assume that $C=1$. Hence, we have $0<F(t,z)\leq 1$ for all $(t,z)\in\mathbb{R}\times\O$. Now we need the following result, which is inspired by Lemma 1.6 in \cite{ChaZe24} for the class of $m$-subharmonic functions. For $\phi\in C^0(\pa\O)$, by $SH_{\om,m}^{\phi}(\O)$ we denote the set of functions $\chi\in SH_{\om,m}(\O)\cap L^{\infty}(\O)$ such that $\chi=\phi$ on $\pa\O$, i.e for any $x\in\pa\O$, $\lim\limits_{\O\ni z \to x}\chi(z)=\phi(x)$.
\begin{lemma}\label{bd34}
	Let $\phi\in C^0(\pa\O)$ and $\chi\in SH_{\om,m}^{\phi}(\O)$. Then there exists a decreasing sequence $\{\chi_j\}_{j\geq 1}$ of continuous functions in $SH_{\om,m}^{\phi}(\O)\cap C^0(\overline{\O})$ which converges to $\chi$ point-wise on $\O$.	
\end{lemma} 
\begin{proof}
	Take a decreasing sequence of continuous functions $\{ h_j\}_{j\geq 1}$ on $\overline{\O}$ which is convergent to $\chi$ on $\overline{\O}$. We can arrange so that $h_j=\phi$ on $\pa\O$. Indeed, by $\O$ is a ball in $\cn$ then there exists a harmonic function $G$ on $\O$ such that $G=\phi$ on $\pa\O$. Because $\chi\in SH_{\om,m}(\O)$ and for all $x\in\pa\O$, $\lim\limits_{\O\ni z\to x}\chi(z)=\phi(x)$ and $G|_{\pa\O}=\phi$ then it follows that $\chi\leq G$ on $\overline{\O}$. Put $a_j(z)=\min\{h_j(z), G(z)\}, z\in\overline{\O}$. Then $a_j$ is continuous on $\overline{\O}$, $\{a_j\}$ is a decreasing sequence and $a_j\searrow \chi$ on $\overline{\O}$. Moreover, from $h_j|_{\pa\O}\geq \chi|_{\pa\O}=\phi$ then it is not difficult to see that $a_j|_{\pa\O}= \phi$. Next, put 
	$$ \chi_j=\sup\{v\in SH_{\om,m}(\O): v\leq a_j,\\ \text{on $\overline{\O}$}\}.$$
	
	\n Then for all $j\geq 1$ we have $\chi_j\geq \chi$ on $\O$. On the other hand, by (e) of Proposition \ref{prop:closure-max} it implies that $\chi_j^{*}\in SH_{\om,m}(\O)$, $\chi_j^{*}\leq a_j$ on $\O$. Hence, $\chi_j^{*}= \chi_j$ on $\O$ and, therefore, $\chi_j\in SH_{\om,m}(\O)$. Since $a_j\searrow \chi$ then $\chi_j\searrow \chi$ on $\O$. It is clear that $\chi\leq \chi_j\leq a_j$ on $\O$ then it follows that $\lim\limits_{\O\ni z\to x} \chi_j(z)= \phi(x)$ for all $x\in\pa\O$. By using arguments as in the proof of Lemma 3.1 in \cite{BeZe}, we claim that all $\chi_j$ for $j\geq 1$ are continuous on $\O$. The proof of Lemma \ref{bd34} is complete.
\end{proof}

\n Now by using Lemma 14 in \cite{HQ23} together with Lemma \ref{bd34} we will complete the proof of the continuity of the solution $u$ of equation \eqref{eq1.5}. From equation \eqref{eq1.5} we know that $u\in SH_{\om,m}(\O)\cap L^{\infty}(\overline{\O})$ and $\lim\limits_{\O\ni z\to x}u(z)=\phi(x)$ for all $x\in\pa\O$. Without loss of generality, we may assume that there exists $M>0$ such that for all $z\in\overline{\O}$, we have $-M \leq u(z)\leq 0$. By Lemma \ref{bd34} we can find a decreasing sequence of continuous functions $u_j\in SH_{\om,m}(\O)\cap C^0(\overline{\O})$, $\lim\limits_{\O\ni z\to x}u_j(z)=\phi(x)$ for all $x\in\pa\O$ for all $j\geq 1$ and $u_j\searrow u$ on $\Omega$. Because $u_j\searrow u$ on $\O$ then we may assume that $-M\leq u_j\leq 0$ on $\O$ for $j\geq 1$. On the other hand, since $(\om+ dd^c u)^m\wed\beta^{n-m}=F(u,z)d\mu\leq d\mu$ and $\mu$ is $\gamma$-Dini-diffuse  with respect to $(\om,m)$-capacity then so is $(\om+ dd^c u)^m\wed\beta^{n-m}$. Next, by applying Lemma 14 in \cite{HQ23} to $u_j$ and $u$ we get that
\begin{equation}\label{eq3.6}
\sup\limits_{\overline{\O}}(u_j-u)\leq C.H_{\gamma}\bigg(\int\limits_{\O}(u_j-u)(\om+dd^c u)^m\wedge\beta^{n-m}\bigg),
\end{equation}

\n where $C=C(M)>0$ is a constant independent on $u_j$ and $H_{\gamma}$ is an inverse function of the function $s\mapsto s^{m(m+2)+1}F_{\gamma}^{-1}(s^{m+2})$ as in the statement of Lemma 14 in \cite{HQ23}. Since $ u_j \searrow u$ as $ j \to \infty$, by the Lebesgue Monotone Convergence Theorem, the right-hand side of \eqref{eq3.6} tends to 0. Hence, $u_j$ is uniformly convergent to $u$ on $\overline{\O}$ and it follows that $u$ is continuous on $\overline{\O}$. The proof of Theorem \ref{th1.1} is now complete.

\end{document}